\documentclass[11pt]{article}
\usepackage{amsthm,amssymb,epsfig,graphicx,amscd,amsmath,color,bm}
\usepackage{graphicx, color,enumerate}
\usepackage{hyperref}
\usepackage[all]{xy}
\theoremstyle{definition}
\newtheorem{theorem}{Theorem}[section]
\newtheorem*{Theorem}{Main Theorem}

\newtheorem{prop}[theorem]{Proposition}

\newtheorem{cor}[theorem]{Corollary}

\newtheorem{example}[theorem]{Example}

\newtheorem{open problem}[theorem]{Open Problem}

\newtheorem{definition}[theorem]{Definition}

\newtheorem{remark}[theorem]{Remark}
\newtheorem{lemma}[theorem]{Lemma}

\newcommand{\R}{\mathbb{R}}
\newcommand{\RR}{\mathbb{R}}
\newcommand{\Z}{\mathbb{Z}}

\newcommand{\Q}{\mathbb{Q}}

\newtheorem{notation}[theorem]{Notation}

\begin{document}

\title{BOUNDARY OF THE RELATIVE OUTER SPACE}
\author{Erika Meucci}
\maketitle

\begin{abstract}
Let $\mathcal{A} = \{ A_1, \ldots, A_k \}$ be a system of free factors of $F_n$. The
group of relative automorphisms $\mathrm{Aut}(F_n; \mathcal{A})$
is the group given by the automorphisms of $F_n$ that restricted
to each $A_i$ are conjugations by
ele\-ments in $F_n$. The group of relative outer automorphisms is defined as
$\mathrm{Out}(F_n; \mathcal{A}) = \mathrm{Aut}(F_n; \mathcal{A}) / \mathrm{Inn}(F_n)$,
where $\mathrm{Inn}(F_n)$ is the normal subgroup of $\mathrm{Aut}(F_n)$ given by all
the inner automorphisms. This group acts on the relative outer space $\mathrm{CV}_n(\mathcal{A})$.
We prove that the dimension of the boundary of the relative outer space
is $\mathrm{dim}(\mathrm{CV}_n(\mathcal{A}))-1$.
\end{abstract}

\section{Introduction}

Let $F_n$ denote the free group of rank $n$. We consider the
group of automorphisms of $F_n$, denoted by $\mathrm{Aut}(F_n)$, and the
group of outer automorphisms
$$
\mathrm{Out}(F_n) = \mathrm{Aut}(F_n) / \mathrm{Inn}(F_n),
$$
where $\mathrm{Inn}(F_n)$ is the normal subgroup of $\mathrm{Aut}(F_n)$ given by all
the inner automorphisms. In 1986
Culler and Vogtmann introduced a space $\mathrm{CV}_n$ on which the
group $\mathrm{Out}(F_n)$ acts with finite stabilizers and proved that $\mathrm{CV}_n$
is contractible. That space $\mathrm{CV}_n$ is called \emph{outer space}.
In \cite{GaLe}, Gaboriau and  Levitt computed the dimension of the boundary
of outer space.
We are interested in studying particular subgroups of $\mathrm{Out}(F_n)$.

Let $\mathcal{A} = \{ A_1, \ldots, A_k \}$ be a system of free factors of $F_n$, that is, there exists
$B < F_n$ such that $F_n = A_1 * \cdots * A_k * B$. We define the
group of relative (to $\mathcal{A}$) automorphisms
$\mathrm{Aut}(F_n; \mathcal{A})$ given by the elements $f \in \mathrm{Aut}(F_n)$
such that $f$ restricted to each $A_i$ is a conjugation by an
element in $F_n$.
Note that $\mathrm{Aut}(F_n) > \mathrm{Aut}(F_n; \mathcal{A}) \triangleright \mathrm{Inn}(F_n)$.
We define also the group of relative outer automorphisms:
$$
\mathrm{Out}(F_n; \mathcal{A}) = \mathrm{Aut}(F_n; \mathcal{A}) / \mathrm{Inn}(F_n) <
\mathrm{Out}(F_n).
$$
In \cite{M}, it was introduced a contractible space, that we will denote by
$\mathrm{CV}_n(\mathcal{A})$, on which
$\mathrm{Out}(F_n; \mathcal{A})$ acts. This space is called relative
outer space and it can be thought as a subset of the compactification of
$\mathrm{CV}_n$. In \cite{M}, we proved that
an irreducible relative outer automorphism with irreducible
powers acts on the compactification of the modified
relative outer space with north-south dynamics.
A natural question arose: what is the dimension of the boundary?
Our goal is to give an answer to that question.
\begin{Theorem}\label{mainthm}
The dimension of $\partial \mathrm{CV}_n(\mathcal{A})$ is equal to
$\mathrm{dim} (\mathrm{CV}_n(\mathcal{A}))-1$.
\end{Theorem}
The main ingredient in the proof of the Main Theorem is the computation of the $\Q$-rank of a
tree (with special points).
In Section 2, we start defining the relative outer space and then
we review some of its properties studied in \cite{M}. In Section 3, we give
the definition of geometric and non-geometric trees with special points
and we discuss some features of the geometric trees that are a relative
version of the facts in \cite{GaLe}. The goal of
Section 4 is to define the index of a tree with special points. An upper
bound of the index will be a key step in the computation of the $\Q$-rank of
a tree with special points and hence in the proof of the Main Theorem in
Section 5.

\subsection*{Acknowledgments} I wish to thank Mladen Bestvina for his suggestions
and very useful conversations.

\section{Relative Outer Space}

In this section we define the relative outer space and its boundary. Moreover,
we review some of the properties of the relative outer space studied in \cite{M}.

Let $\mathcal{A} = \{ A_1, \ldots, A_k \}$ be a system of free factors of $F_n$, that is, there exists
$B < F_n$ such that $F_n = A_1 * \cdots * A_k * B$.

Consider $A_i=<y_1^i, \ldots,
y_{s(i)}^{i}>$ and $F_n=<y_1^1, \ldots, y_{s(k)}^{k}, x_1, \dots,
x_{n-\sum_{i=1}^{k} s(i)}>$. By a graph we mean a connected $1$-dimensional CW complex.
Let the \emph{relative rose} $R_n(\mathcal{A})$ be a graph obtained by a wedge of $n-\sum_{i=1}^{k}
s(i)$ circles $e_1, \ldots, e_{n-\sum_{i=1}^{k} s(i)}$ attaching
$\sum_{i=1}^{k} s(i)$ circles $C_1^1, \ldots, C_{s(k)}^{k}$ on $k$ stems,
where the $i$th stem has $C_1^i, \ldots, C_{s(i)}^{i}$ attached to it.
Moreover,
$$
\pi_1(R_n(\mathcal{A}), v) \cong F_n = <y_1^1, \ldots,
y_{s(k)}^{k}, x_1, \dots, x_{n-\sum_{i=1}^{k} s(i)}>,
$$
where $v$ is the vertex in $R_n(\mathcal{A})$ intersection
of the circles $e_1, \ldots, e_{n-\sum_{i=1}^{k} s(i)}$,
by declaring $y_i^j$ to be the homotopy
class of $C_i^j$ and $x_i$ to be the homotopy class of the loop
$e_i$.
Let $(R_n(\mathcal{A}), \underline{k})$ be the graph
$R_n(\mathcal{A})$ equipped with inclusions $k_j: \bigvee_{i=1}^{s(j)} S^1
\rightarrow R_n(\mathcal{A})$ that identifies
$\bigvee_{i=1}^{s(j)} S^1$ with $\bigvee_{i=1}^{s(j)} C_i^j$, for
all $j=1, \ldots, k$.
\begin{definition}
Let $\Gamma$ be a graph of rank $n$ with vertices of valence at
least $3$, equipped with embeddings $l_j: \bigvee_{i=1}^{s(j)} S^1
\rightarrow \Gamma$ for $j=1, \ldots, k$. We call $\mathbb{B}_j =
l_j(\bigvee_{i=1}^{s(j)} S^1)$ \emph{wedge cycle}. The \emph{dual
graph} of the $\mathbb{B}_j$'s is the graph with one vertex for each
wedge cycle, one vertex $w$ for each intersection between two or
more wedge cycles and edges between $w$ and vertices corresponding
to the wedge cycles meeting in $w$.
\end{definition}
\begin{definition}
An $(\mathcal{A},n)$\emph{-graph} $(\Gamma, \underline{l})$ is a
finite graph $\Gamma$ of rank $n$ with vertices of valence at least $3$,
with possible separating edges, equipped with embeddings $l_j:
\bigvee_{i=1}^{s(j)} S^1 \rightarrow \Gamma$ for $j=1, \ldots, k$,
such that any two $\mathbb{B}_j$ intersect in at most a point and
the dual graph of the $\mathbb{B}_j$'s is a forest.
\end{definition}

\begin{notation}
We denote by $\mathcal{A}$ the set of free factors $A_1,\ldots, A_k$
and we will denote $A_1 * \cdots * A_k * B$ by $\mathcal{A}*B$.
Sometimes we will also write $\mathcal{A}$ for the subgroup
$A_1 * \cdots * A_k$. The meaning of $\mathcal{A}$ should be clear from the context.
\end{notation}

\begin{definition}
A \emph{marked metric $(\mathcal{A},n)$-graph} $(\Gamma,\phi, l)$
is a marked graph (with possible separating edges) $(\Gamma, \phi)$ such that
\begin{itemize}
\item each edge $e$ in $\widehat{\Gamma}$ (the graph obtained from $\Gamma$ by collapsing the wedge
cycles to special points) has length $\widehat{l}(e) = l_{|\widehat{\Gamma}}(e)
\in (0,1]$ and each edge in a wedge cycle has length $0$;
\item the wedge cycles are disjoint.
\end{itemize}
\end{definition}
\begin{definition} \label{outspace}
The \emph{relative outer space} $\mathrm{CV}_n(\mathcal{A})$ is the space
of equivalence classes of marked metric $(\mathcal{A},n)$-graphs
where
        \begin{enumerate}
        \item the sum of all lengths of the edges in $\widehat{\Gamma}$
        is $1$ (relative volume $1$);
        \item $(\Gamma_1,\phi_1,l_1) \sim (\Gamma_2,\phi_2,l_2)$ if there is
        an isometry $h:\Gamma_1 \rightarrow \Gamma_2$ such that $h \circ
        \phi_1(C_i^j)=\phi_2(C_i^j)$, $\forall \, i,j$ and $h \circ \phi_1$
        is homotopic to $\phi_2$ rel. $C_i^j$, $\forall \,i,j$.
        \end{enumerate}
\end{definition}
\begin{remark}
This space (modulo possible edges of length $0$ not in the wedge cycles)
was introduced in the third chapter of \cite{M}, and it was called
modified relative outer space to distinguish this space from another space,
called relative outer space, introduced in the second chapter.
\end{remark}

Let $R_{n}(\mathcal{A})$ be the relative rose with each edge in
the wedge cycles of length $0$.
There is a natural right action of $\mathrm{Out}(F_n; \mathcal{A})$ on $\mathrm{CV}_n(\mathcal{A})$
given by changing the marking:
let $X = (\Gamma,\phi,l) \in \mathrm{CV}_n(\mathcal{A})$
and $\Psi \in \mathrm{Out}(F_n; \mathcal{A})$,
consider $\psi: R_{n}(\mathcal{A}) \rightarrow R_{n}(\mathcal{A})$
such that it fixes the wedge cycles and $[\psi_*] = \Psi$. The right action is given by
$$
X \cdot \Psi =(\Gamma,\phi,l) \cdot \Psi = (\Gamma, \phi \circ \psi,l).
$$
Notice that the stabilizer of a point may be infinite.
We define a topology on $\mathrm{CV}_n(\mathcal{A})$ by varying the length
of the edges that are not in any wedge cycle. Since the sum of the lengths of
these edges is $1$, $\mathrm{CV}_{n}(\mathcal{A})$ is a simplicial complex with
missing faces.
We define the \emph{relative spine} $S_{n}(\mathcal{A})$
of the relative outer space as the geometric realization of the
partially ordered set of open simplices. Notice that $S_{n}(\mathcal{A})$
is a simplicial complex.

In \cite{M}, we proved the following result.
\begin{theorem}\label{contramodispace}
The relative outer space $\mathrm{CV}_n(\mathcal{A})$ is contractible.
\end{theorem}

In \cite{M}, we computed the dimension of $\mathrm{CV}_n(\mathcal{A})$
and the dimension of $S_{n}(\mathcal{A})$.
We briefly sketch the proof of such computation.
Suppose that $\Gamma$ is a maximal graph in $S_{n}(\mathcal{A})$,
i.e., it has the maximum number of vertices,
and consider $\widehat{\Gamma}$. Denote by $V$ and $E$ the number
of vertices and edges of $\widehat{\Gamma}$ respectively.
The vertices corresponding to the special points have
valence $1$ and the remaining vertices have valence $3$. Hence,
$$
E = \frac{3(V-k)}{2} + \frac{k}{2}.
$$
Because $V-E = 1 - (n - \sum_{i=1}^{k} s(i))$, we have
$$
V = 2n + 2k - 2 - 2 \sum_{i=1}^{k} s(i)
$$
and hence, because we can collapse $V$ vertices to $s = \max \{ k, 1 \}$ vertices
and the wedge cycles are disjoint,
$$
\mbox{dim} (S_{n}(\mathcal{A})) = 2n + 2k - 2 - 2 \sum_{i=1}^{k} s(i) - s.
$$
Moreover,
$$
E = 3n + 2k - 3 - 3 \sum_{i=1}^{k} s(i)
$$
and because the relative volume of each graph is $1$,
$$
\mbox{dim} (\mathrm{CV}_n(\mathcal{A})) = 3n + 2k - 4 - 3 \sum_{i=1}^{k} s(i).
$$
Notice that if $k=0$, $\mbox{dim} (\mathrm{CV}_{n}(1)) = 3n-4 = \mbox{dim} (\mathrm{CV}_{n})$.
Indeed, if $k=0$ the relative outer space is the standard outer space.

\begin{example}\label{expoly}
Consider $\mathrm{Out}(F_2;A)$, where $F_2 = <a, \, b>$,
$A=<a>$.
In that case, $\mathrm{Out}(F_2;A)$ is isomorphic to the infinite dihedral
group $D_{\infty}$ (see \cite{M}).
The relative outer space $\mathrm{CV}_{2}(A)$ is a point $X$ with
an infinite countable number of half-open edges attached to it.
The action of the group on $\mathrm{CV}_{2}(A)$
is given by rotating the edges. Hence, the stabilizer of $X$ is $\mathrm{Out}(F_2;A)$.
Moreover, $\mathrm{CV}_{2}(A)/ \mathrm{Out}(F_2;A)$ is a $1$-simplex with a missing vertex.
The relative spine $S_{2}(A)$ is a point with an infinite number of
closed $1$-simplices coming out from that point. The relative outer space
$\mathrm{CV}_{2}(A)$ is not locally compact and this is true in general:
$\mathrm{CV}_n(\mathcal{A})$ is not locally compact if $\mathcal{A} \neq 1$.
\end{example}

We can think of the relative outer space as a deformation space of trees.
We recall the main definitions and some results for actions on $\R$-trees.
\begin{definition}
Let $(X,d)$ be a metric space.
We say that $(X,d)$ is an $\R$-\emph{tree} if for any $x,y \in X$ there is a
unique arc from $x$ to $y$ and this arc is a geodesic segment.
\end{definition}
Let $\phi: T \rightarrow T$ be an isometry of an $\R$-tree $T$. The \emph{translation length}
of $\phi$ is
$$
l(\phi) = \inf \{ d(x, \phi(x)) \, | \, x \in T \}.
$$
The infimum is always attained and there are two possible cases.
If $l(\phi)>0$, there is a unique $\phi$-invariant
line called the axis of $\phi$, and $\phi_{| \text{axis}}$ is a translation by $l(\phi)$.
In this case, we say that $\phi$ is \emph{hyperbolic}.
If $l(\phi) = 0$, then $\phi$ fixes a non-empty subtree of $T$ and is said
to be \emph{elliptic}.

Let $G$ be a group acting by isometries on an $\R$-tree $T$. A tree
equipped with an isometric action is called $G$-\emph{tree}.
The action is \emph{nontrivial} if no point of $T$ are fixed by the whole group.
It is \emph{minimal} if there is no proper $G$-invariant subtree. The action is
\emph{free} if any nonidentity group element does not leave an element of $T$ fixed.
Let $Gx = \{ gx \, | \,  g \in G \}$ be the orbit of $x \in T$.
An action of $G$ on $T$ has \emph{dense orbits} if the closure of $Gx$ is the whole tree $T$.
A map $f:T \rightarrow T'$ between $\R$-trees is a \emph{morphism} if every segment in $T$
can be written as a finite union of subsegments, each of which is mapped
isometrically into $T'$. Note that an equivariant morphism does not increase distances.

The notion of a deformation space was introduced by Forester in \cite{F}.
By definition, two $G$-trees are in the same deformation space if they have
the same elliptic subgroups, i.e., if a subgroup of $G$ fixes one point in a tree, it also
fixes the image of that point in any other tree. Identifying two trees if they differ
only by rescaling the metric leads to the projectivized deformation space.
Guirardel and Levitt \cite{GL} and Clay \cite{Clay} proved the contractibility of this space.
The relative outer space $\mathrm{CV}_{n}(\mathcal{A})$ is a projectivized deformation space.

Let the \emph{deformation space} $\mathcal{D}$ be the space of simplicial $F_n$-trees with elliptic subgroups
$A_1, \ldots, A_k$.
Let $X=(\Gamma, \phi,l) \in \mathrm{CV}_{n}(\mathcal{A})$. The tree $T_1$ associated to $X$
is constructed in the following way. Let $\Gamma_0$ be the graph obtained by $\Gamma$
changing the length of the wedge cycles from $0$ to a constant $\varepsilon>0$.
Consider the universal cover $\widetilde{\Gamma_0}$ of $\Gamma_0$ and collapse all the
rays that correspond to words $a_{i_1} a_{i_2} a_{i_3} \cdots$,
$a_{i_j} \in \mathcal{A}$ and its translates.
We will call $T_1$ an $(\mathcal{A}*B)$-tree with special vertices or just
$(\mathcal{A}*B)$-tree.
\begin{example}\label{extree1}
Consider $\mathrm{Out}(F_2;A)$, where $F_2 = <a,b>$,
$A=<a>$, and consider the point $(\Gamma,\phi = \mathrm{id},l)
\in \mathrm{CV}_2(A)$ consisting of a loop corresponding to $b$ and
a vertex corresponding to $a$. The graph $\Gamma_0$ is a rose with
two petals. One petal corresponds to $a$ and its length is $\varepsilon$
and the other petal correspond to $b$ and its length is $1$.
In order to construct the tree $T_1$ associated to this graph, first we consider
the universal covering of $\Gamma_0$, and
then we collapse the $a$-axis and its translated axes.
Notice that there are infinitely many edges coming out from
each vertex (see Figure~\ref{fig:extree1bo}).

\begin{figure}[htbp]
\begin{center}
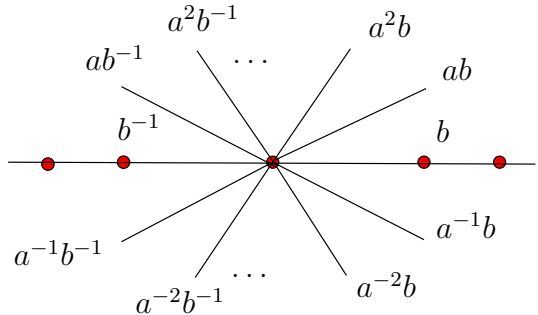

\caption{The tree $T_1$ associated to $(\Gamma,\phi,l)$ in
Example~\ref{extree1}.} \label{fig:extree1bo}
\end{center}
\end{figure}

\end{example}
In \cite{M}, we defined the boundary of the relative outer space.
We recall the following definition due to Cohen and Lustig.
\begin{definition}
An action of $F_n$ on an $\R$-tree is \emph{very small} if
\begin{enumerate}
\item all edge stabilizers are cyclic (\emph{small action}),
\item $\mbox{Fix}(g)$ is isometric to a subset of $\R$ for $1 \neq g \in F_n$
(no fixed tripods), and
\item $\mbox{Fix}(g) = \mbox{Fix}(g^i)$ for all $i \geq 2$ (no obtrusive powers).
\end{enumerate}
\end{definition}
Cohen and Lustig \cite{CoLu} showed that a simplicial action is in $\overline{\mathrm{CV}_n}$ if and
only if it is very small. Bestvina and Feighn \cite{BFouter} proved that this is actually true for
all actions concluding that the closure of outer space is the set of very small
actions of $F_n$ on $\R$-trees.
In the relative case, the main issue is understanding on which $\R$-trees
the group $\mathcal{A} *B$ is acting on and those trees are the $(\mathcal{A} *B)$-trees
with special points.
Notice that $\mathrm{CV}_n(\mathcal{A})$ embeds naturally
in the space $\mathfrak{T}$ of actions of
$\mathcal{A}*B$ on metric $\R$-trees with special points
such that all edge stabilizers are cyclic.
As in the case of the standard outer space $\mathrm{CV}_n$, we can endow $\mathrm{CV}_n(\mathcal{A})$
with three topologies: the simplicial (or weak) topology, the length function (or axes)
topology, and the Gromov-Hausdorff topology. We already defined the simplicial
topology. The length function topology is the coarsest topology making the translation length
function $T \mapsto l_T(\phi)$ continuous. The Gromov-Hausdorff topology
is defined in the following way. Given $X \subset T$, $A \subset \mathrm{Out}(F_n; \mathcal{A})$,
and $\varepsilon >0$, a fundamental system of neighborhoods for $T$,
denoted by $U_T(X,A, \varepsilon)$ is the set of trees $T'$ such that there is
a map $f:X \rightarrow T'$ satisfying
$$
| d(x, \phi \cdot y)-d(f(x), \phi \cdot f(y)) | < \varepsilon, \quad \forall \, x,y \in X, \, \forall \, \phi \in A.
$$
In the standard $\mathrm{CV}_n$ the three topologies are equivalent (see \cite{P}).
Because the trees corresponding to the points in $\mathrm{CV}_n(\mathcal{A})$
are \emph{irreducible} (in the sense of \cite{GL}), $\mathrm{CV}_n(\mathcal{A})$ with
the length function topology is homeomorphic to $\mathrm{CV}_n(\mathcal{A})$
with the Gromov-Hausdorff topology. However, the length function topology of
the embedding in the space of actions is not the simplicial topology of
$\mathrm{CV}_n(\mathcal{A})$. For example,
in the case $n=2$, $F_2=<a,b>$ and $A=<a>$, the modified relative outer space
$\mathrm{CV}_2(A)$ is the union of half-open $1$-simplices attached to a point $X$
(see Example~\ref{expoly}). The relative outer space $\mathrm{CV}_2(A)$ endowed
with the simplicial topology is not (locally) compact, while $\mathrm{CV}_2(A)$
endowed with the length function topology is a compact space: the two sequences
of trees that are the middle-points of the $1$-simplices corresponding to the graphs
with marking induced by $a \mapsto a$, $b \mapsto a^N b$,
as $N \rightarrow \pm \infty$, converge to $X$.
We can consider the closure $\overline{\mathrm{CV}_n}(\mathcal{A})$
of the image of the embedding $\mathrm{CV}_n(\mathcal{A}) \hookrightarrow \mathfrak{T}$.
The closure consists of projective classes of actions on
metric $\R$-trees with special points where
\begin{enumerate}
\item all edge stabilizers are cyclic,
\item $\mbox{Fix}(g)$ is isometric to a subset of $\R$ for $1 \neq g \in F_n / \mathcal{A}$,
\item the $A_i$'s are elliptic elements, and
\item $\mbox{Fix}(g) = \mbox{Fix}(g^i)$ for all $i \geq 2$.
\end{enumerate}
The \emph{boundary of the relative outer space} is
$$
\partial \mathrm{CV}_n(\mathcal{A}) = \overline{\mathrm{CV}_n}(\mathcal{A}) \setminus \mathrm{CV}_n(\mathcal{A}).
$$
Our goal is to prove that the dimension of $\partial \mathrm{CV}_n(\mathcal{A})$
is the dimension of $\mathrm{CV}_n(\mathcal{A}) -1$.
From now on, when we talk about
the topology of $\mathrm{CV}_n(\mathcal{A})$, $\overline{\mathrm{CV}_n}(\mathcal{A})$,
or $\partial \mathrm{CV}_n(\mathcal{A})$ we mean the length function topology.
In this topology, $\mathrm{CV}_n(\mathcal{A})$ is contractible
and $\overline{\mathrm{CV}_n}(\mathcal{A})$ is contractible (see \cite{GL}),
and compact (because it is a closed subset of the compact set $\overline{\mathrm{CV}_n}$).
However, $\partial \mathrm{CV}_n(\mathcal{A})$ is not compact (see Example~\ref{expoly}
and the above discussion about the two sequences).

\section{(Non-)Geometric Trees with Special Points}

Let $T$ be a minimal $(\mathcal{A}* B)$-tree with special points and length function $l$.
A subtree of $T$ is called \emph{finite} if it is the convex hull of a finite
subset.
Let $K \subset T$ be a finite subtree such that $K \cap x_i K \neq \emptyset$, for
$i \in \{ 1, \ldots, n - \sum s(i) \}$ and $K \cap A_j K \neq \emptyset$, for
any $j \in \{ 1, \ldots, k \}$.
Let $\mathcal{K}=(K, \{ \varphi_i, \varphi^j \})$ be the system with $\varphi_i$
the restriction of the action of $x_i$ to $x_{i}^{-1} K \cap K$ and $\varphi^j$
the restriction of the action of $A_j$ to $A_{j}^{-1} K \cap K$.
By Theorem I.1 \cite{GaLe}, there exists a unique $(\mathcal{A}* B)$-tree such that
\begin{enumerate}
\item $T_{\mathcal{K}}$ contains $K$ as an isometrically embedded subtree;
\item If $p \in P_i= x_{i}^{-1} K \cap K$, then $x_ip= \varphi_i(p)$.
If $p \in Q_j = A_{j}^{-1} K \cap K$ and $p$ is not a special point
with elliptic subgroup $A_j$, then there is $a \in A_j$ such
that $p \in a^{-1} K \cap K$, and $a p = \varphi^j(p)$.
If $p$ is a special point with elliptic subgroup $A_j$, $\varphi^j(p)=p$;
\item every orbit of the action meets $K$;
\item if $T'$ is another $(\mathcal{A}* B)$-tree satisfying the first two
items, there exists a unique equivariant morphism $f: T_{\mathcal{K}} \rightarrow T'$
such that $f(p)=p$, for $p \in K$.
\end{enumerate}

If the action on $T'$ is (very) small, then the action on $T_{\mathcal{K}}$ is
(very) small.
Moreover, the tree $T_{\mathcal{K}}$ is not necessarily minimal, but there are arbitrarily
large subtrees $K$ such that $T_{\mathcal{K}}$ is minimal (see Section II in
\cite{GaLe}).

\begin{definition}
If there exists $\mathcal{K}$ such that $T_{\mathcal{K}} = T$, then $T$ is called
\emph{geometric}. Otherwise, $T$ is called \emph{non-geometric}.
\end{definition}

\begin{prop}\label{PropI4}
Let $T$ be a geometric minimal $(\mathcal{A}* B)$-tree, i.e., $T= T_{\mathcal{K}}$
for some system $\mathcal{K}$. We have:
\begin{enumerate}
\item Two points $(p_1, \omega_1),(p_2, \omega_2) \in K \times (\mathcal{A}* B)$
define the same point in $T_{\mathcal{K}}$ if and only if
$$
p_2 = \varphi_{i_1}^{\varepsilon_1} \cdots \varphi_{i_m}^{\varepsilon_m}(p_1) \mbox{ such that }
\omega_{2}^{-1} \omega_1 = \alpha_{i_1}^{\varepsilon_1} \cdots \alpha_{i_m}^{\varepsilon_m},
$$
where $\varphi_{i_j} \in \{ \varphi_i, \varphi^j \}$, $\varepsilon_{i_j} \in \{ \pm 1 \}$,
$\alpha_{i_j} \in \{ y_1^1, \ldots, y_{s(k)}^{k}, x_1, \dots, x_{n-\sum_{i=1}^{k} s(i)} \}$.

\item Let $p_1,p_2 \in K$ and $\omega \in \mathcal{A}* B$. We have $p_2 = \omega p_1$
if and only if
$$
p_2 = \varphi_{i_1}^{\varepsilon_1} \cdots \varphi_{i_m}^{\varepsilon_m}(p_1) \mbox{ with }
\omega = \alpha_{i_1}^{\varepsilon_1} \cdots \alpha_{i_m}^{\varepsilon_m},
$$
where $\varphi_{i_j}$, $\varepsilon_{i_j}$, and $\alpha_{i_j}$
are as above.
\end{enumerate}
\end{prop}

Let $S$ be the finite set of vertices of $K$, $x_{i}^{-1} K \cap K$,
$A_{j}^{-1} K \cap K$, $\varphi_i(x_{i}^{-1} K \cap K)$,
$\varphi^j(A_{j}^{-1} K \cap K)$, for all values of $i$ and $j$.
We can list some properties of $T_{\mathcal{K}}$.
See \cite{GaLe} for a proof of these facts.
\begin{prop}\label{PropI8andCor}
Let $T_{\mathcal{K}}$ as above.
\begin{enumerate}
\item If $p \in T_{\mathcal{K}}$ is a branch point, its orbit contains a
point of $S$. The action of $\mathrm{Stab}(p)/ \mathcal{A}$ on the set of directions
$\pi_0(T_{\mathcal{K}} \setminus \{ p \})$ has only finitely many orbits.
In particular, there are only finitely many orbits of branch points in
$T_{\mathcal{K}}$.

\item If the $(\mathcal{A}* B)$-action on $T_{\mathcal{K}}$ is small with
dense orbits, then every edge stabilizer is trivial.
\end{enumerate}
\end{prop}

\section{The Index of a Tree with Special Points}

We introduce the index of a tree with special points and we find an upper bound
for such index. This upper bound will be essential in the computation of the
$\Q$-rank of a tree with special points.

Let $T$ be a minimal $(\mathcal{A}* B)$-tree with special points. Given $x \in T$,
a \emph{direction} $d$ from $x$ is a germ of isometric embeddings
$d: [0, \varepsilon] \rightarrow T$ with $d(0)=x$. The subgroup
$\mathrm{Stab}(x) < \mathcal{A}* B$ acts on the set
$$
D = \{ d \, | \, d \mbox{ is a direction from } x \}.
$$
Note that if $d \in D$, then $\mathrm{Stab}(d)$ is either trivial
or infinite cyclic.
\begin{notation}
Let $x \in T$. We denote
$$
\mbox{rk}(\mathrm{ST}(x)) =
\left\{
  \begin{array}{l}
    \mbox{rk}(\mathrm{Stab}(x) / \mathcal{A}), \mbox{ if } x \mbox{ is not a special point}, \\
    \mbox{rk}(\mathrm{Stab}(x) / \mathcal{A}) +1, \mbox{ if } x \mbox{ is a special point}. \\
  \end{array}
\right.
$$
\end{notation}
\begin{definition}
The \emph{index} of $x \in T$ is
$$
i(x) = 2 \mbox{rk}(\mathrm{ST}(x)) + v_1(x)-2,
$$
where $v_1(x)$ is the number of $\mathrm{Stab}(x)$-orbits of directions
from $x$ with trivial stabilizer.
\end{definition}

Note that the definition coincides with the definition in
\cite{GaLe} if there are no special points, i.e., if $\mathcal{A}=1$.
We will prove that $i(x)$ is finite
for any $x \in T$.
Notice that if $\mathrm{Stab}(x) = 1$, then $i(x)+2$ is the cardinality
of $D$. Moreover, we have the following result.

\begin{prop}\label{PropIIIone}
The index $i(x)$ is always non-negative. If $i(x)>0$, then $x$ is a
branch point. Conversely, if the action
is very small, then every branch point that is not a special point
has $i(x) \geq 1$.
\end{prop}

The proof of Proposition~\ref{PropIIIone} is similar to the proof of
Proposition III.1 \cite{GaLe} considering the three cases:
$\mbox{rk}(\mathrm{ST}(x)) \geq 2$, $\mbox{rk}(\mathrm{ST}(x)) = 0$,
and $\mbox{rk}(\mathrm{ST}(x)) = 1$.

\begin{remark}
If $\mathcal{A} \neq 1$, for any special point $x \in T$, there is
$i \in \{ 1, \ldots, k \}$ such that $A_i < \mathrm{Stab}(x)$,
and $\mbox{rk}(\mathrm{Stab}(x)) \geq \mbox{rk}(A_i)$.
\end{remark}

Let $x \in T$. We denote the $(\mathcal{A}*B)$-orbit of $x$ by $[x]$.
Because $i(x_1) = i(x_2)$ for any $x_1, x_2 \in T$ such that $[x_1]=[x_2]$,
the index $i([x])$ is well defined.
\begin{definition}
The \emph{total index} of $T$ is
$$
i(T) = \sum_{[x] \in T/ (\mathcal{A}*B)} i([x]).
$$
\end{definition}

\begin{theorem}\label{thmIII2}
Let $T$ be a minimal small $(\mathcal{A}*B)$-tree.
\begin{enumerate}
\item If $T$ is geometric, then
$$
i(T) = 2n+2k-2-2 \sum_{i=1}^{k} s(i).
$$
\item If $T$ is not geometric, then
$$
i(T) < 2n+2k-2-2 \sum_{i=1}^{k} s(i).
$$
\end{enumerate}
\end{theorem}

\begin{proof}
First, assume that $T$ is geometric. Hence, $T = T_{\mathcal{K}}$,
for some system $\mathcal{K}$. Let $[x] \subset T_{\mathcal{K}}$.
We define the graph $[x]_{\mathcal{K}}$ with

\begin{itemize}
\item vertices $p \in [x] \cap K$;

\item an edge $x_i$ from $p$ to $\varphi_i(p)$ if $p \in P_i$
and a loop $\gamma_j$ if $p$ is elliptic
and $A_j$ is the elliptic subgroup acting non-trivially on $D$.
\end{itemize}

By Proposition~\ref{PropI4}, $[x]_\mathcal{K}$ is connected.
We define the weight of an edge $e$ corresponding to $x_i$ and denoted
by $w(e)$, to be the valence of its origin $p \in P_i$. The weight of
$\gamma_j$ is $1$. All but finitely many edges have
weight $2$. If $\Gamma$ is a finite tree and $p \in \Gamma$, we denote
by $v_{\Gamma}(p)$ the valence of $p$ in $\Gamma$.
We define the graph $[x]_{\mathcal{K}}^{d}$ replacing each
vertex $p$ of $[x]_{\mathcal{K}}$ by $v_{K}(p)$ vertices
representing directions from $p \in K$, and replacing each edge by
$w(e)$ edges. Let $\pi: [x]_{\mathcal{K}}^{d} \rightarrow [x]_{\mathcal{K}}$
be the natural projection.

\begin{lemma}\label{importantlemma1}
If $p \in [x] \cap K$, then the following statements are true.
\begin{enumerate}
\item $\mbox{rk}(\pi_1([x]_{\mathcal{K}})) = \mbox{rk}(\mbox{ST}(p))$.

\item If $x$ is not a special point, $\pi_0([x]_{\mathcal{K}}^{d})$
is in one-to-one correspondence with the set of orbits under
$\mathrm{Stab}(p)/ \mathcal{A}$ of directions $d$ from $p$ in $T_{\mathcal{K}}$.
If $x$ is a special point, $\pi_0([x]_{\mathcal{K}}^{d})$
is in one-to-one correspondence with $\gamma_j$ union the set of orbits under
$\mathrm{Stab}(p)/ \mathcal{A}$ of directions $d$ from $p$ in $T_{\mathcal{K}}$.

\item If $x$ is not a special point and $C$ is a component of
$[x]_{\mathcal{K}}^{d}$, then $\pi_1(C)$ is isomorphic to the
corresponding subgroup $\mathrm{Stab}(d) \in \{ 1, \Z \}$.
If $x$ is a special point and $C$ is a component of
$[x]_{\mathcal{K}}^{d}$ not containing $\gamma_j$, then $\pi_1(C)$ is isomorphic to the
corresponding subgroup $\mathrm{Stab}(d) \in \{ 1, \Z \}$.
\end{enumerate}
\end{lemma}

Lemma~\ref{importantlemma1} follows from Lemma III.5 \cite{GaLe} and the definition of
$[x]_{\mathcal{K}}$. Let $G$ be a finite connected subgraph of $[x]_{\mathcal{K}}$
containing each vertex in $S$ and every edge of weight not equal to $2$.
Let $G' = \pi^{-1}(G) \subset [x]_{\mathcal{K}}^{d}$.
By Proposition~\ref{PropI8andCor} and Lemma~\ref{importantlemma1}, $[x]_{\mathcal{K}}^{d}$
has a finite number of components. Moreover, we can assume that $\pi_1(G'_j)$
generates $\pi_1(C_j)$, for any component $C_j$ of $[x]_{\mathcal{K}}^{d}$.
As in the proof of Theorem III.2 \cite{GaLe}, $\pi_1([x]_{\mathcal{K}})$ is finitely
generated and we may assume that $\pi_1(G)$ generates $\pi_1([x]_{\mathcal{K}})$.
We can consider $[x] \cap S \neq \emptyset$, otherwise $i([x])=0$.

\begin{remark}
If $T'$ is a finite tree and $v_{T'}(p)$ denotes the valence of $p \in T'$,
then
$$
\sum_{p \in T'} (v_{T'}(p)-2) = -2.
$$
\end{remark}

By Lemma~\ref{importantlemma1}, the definition of $G$, and
$2-2 \mbox{rk} (\pi_1(G)) = 2V(G) - 2E(G)$, we have
$$
\left.
  \begin{array}{ccl}
    i([x]) & = & 2 \mbox{rk}(\mbox{ST}(x)) - 2 + v_1(x) = 2 \mbox{rk} (\pi_1([x]_{\mathcal{K}}) -2 + \sum_j (1- \mbox{rk}(\pi_1([x]_{\mathcal{K}}^{d}))) \\
     & = & 2 \mbox{rk} (\pi_1(G)) -2 + \sum_j (1- \mbox{rk}(\pi_1(G'_j))) \\
     & = & \sum_{p \in V(G)} (v_K(p) -2) - \sum_{e \subset E(G)}(w(e)-2) \\
     & = & \sum_{p \in V([x]_{\mathcal{K}})} (v_K(p) -2) - \sum_{e \subset E([x]_{\mathcal{K}})}(w(e)-2) \\
     & = & \sum_{p \in [x] \cap K} (v_K(p) -2) +2k - \sum_{i=1}^{n- \sum s(i)} \sum_{p \in [x] \cap P_i}(v_{P_i}(p)-2). \\
  \end{array}
\right.
$$
Hence,
$$
i(T) = \sum_{[x] \in T/ (\mathcal{A}*B)} i([x])= -2+2k+2n-2 \sum_{i=1}^{k} s(i).
$$
The case when $T$ is non-geometric can be done, as in the proof of Theorem III.2 \cite{GaLe},
appro\-xi\-mating $T$ with a sequence of minimal small $(\mathcal{A}*B)$-trees
$T_{\mathcal{K}_m}$ \emph{strongly converging} to the tree $T$.
\end{proof}

\begin{cor}\label{CorIIIthree}
If $T$ is a minimal very small $(\mathcal{A}*B)$-tree, the number of
orbits of branching points is at most $2n + 2k - 2 - 2 \sum_{i=1}^{k} s(i)$.
\end{cor}

\begin{example}
Consider $\mathrm{Out}(F_2;A)$, where $F_2 = <a,b>$,
$A=<a>$, and the geometric tree $T_1$ in Figure~\ref{fig:extree1bo}.
In this case we have only one orbit of branching points. Let $x \in T_1$
be a branch point.
The stabilizer $\mathrm{Stab}(x) = A =<a>$. Hence,
$\mbox{rk}(\mathrm{Stab}(x)) = \mbox{rk}(A)=1$. Moreover,
$v_1(x) = | \{ [b], [b^{-1}] \} | =2$, where $[b^{\varepsilon}]$ is the
$A$-orbit of the direction associated to $b^{\varepsilon}$
($\varepsilon \in \{ \pm 1 \}$). Therefore,
$$
i(T_1) = i(x) = 2 \cdot 1 +2 - 2 = 2 = 2n+2k-2-2 \sum_{i=1}^{k} s(i),
$$
because $n=2$, $k=1$, and $s(1)=1$.
\end{example}

\section{The $\Q$-rank of a Tree with Special Points}

Let $T \neq \R$ be a minimal $(\mathcal{A}*B)$-tree with non-Abelian
length function $l$. Let $L= L(T)< \R$ be generated by the values of $l$.
The $\Q$-rank of $T$ is denoted by $r_{\Q}(T)$ and is the dimension
of the $\Q$-vector space $L \otimes_{\Z} \Q$ generated by $L$.
If $r_{\Q}(T)$ is finite, $r_{\Q}(T)$ is the rank $r$ of the
Abelian group $L$ and $L/2L \cong (\Z / 2 \Z)^r$.
Let $\Lambda= \Lambda(T)$ be the subgroup generated by distances between the
branch points.

\begin{prop}\label{Prop41}
Let $T \neq \R$ be a minimal $(\mathcal{A}*B)$-tree with non-Abelian
length function and no inversion.
\begin{enumerate}
\item Let $x_1, \ldots, x_{n-\sum_{i=1}^{k} s(i)}$ be a basis for $B$. The numbers
$$
l \Big( x_1 \Big), \ldots, l \Big( x_{n-\sum_{i=1}^{k} s(i)} \Big)
$$
ge\-ne\-ra\-te $L \mbox{ mod } 2 \Lambda$.
\item Let $\{ p_j \}_{j \in J}$ be representatives of $(\mathcal{A}*B)$-orbits
of branch points. For $j_0 \in J$, the numbers of $d(p_{j_0},p_j)$ generate
$\Lambda \mbox{ mod } L$.
\end{enumerate}
\end{prop}

Proposition~\ref{Prop41} is the relative version of Proposition IV.1 \cite{GaLe}.
Indeed, in our case $l(y)=0$, for $y \in A_j$ and $j = 1, \ldots, k$.

\begin{cor}\label{cortoProp41}
The group $\Lambda/ 2 \Lambda$ is generated by $n- \sum_{i=1}^{k} s(i)+b-1$ elements, where $b$ is
the number of orbits of branch points.
\end{cor}

\begin{prop}\label{Prop42}
\begin{enumerate}
\item Geometric $(\mathcal{A} *B)$-actions have finite rank.

\item Consider a non-geometric $(\mathcal{A} *B)$-tree with special points $T$
as the strong limit of a sequence $T_{\mathcal{K}_m}$.
If $\liminf_{m \rightarrow + \infty} r(T_{\mathcal{K}_m})< \infty$,
then
$$
r_{\Q}(T) \leq \liminf_{m \rightarrow + \infty} r(T_{\mathcal{K}_m})
$$
and
$$
r_{\Q}(T) < \liminf_{m \rightarrow + \infty} r(T_{\mathcal{K}_m}).
$$
\end{enumerate}
\end{prop}

Proposition~\ref{Prop42} is the relative version of Proposition IV.2 \cite{GaLe}.

\begin{cor}\label{Cor43}
Let $T$ be a geometric minimal $(\mathcal{A} *B)$-tree without inversions. Let
$b$ the number of orbits of branch points. Then $r(T) \leq n- \sum_{i=1}^{k} s(i)+b -1$.
\end{cor}

\begin{proof}
By Proposition~\ref{Prop42}, the action has finite rank $r$. By Corollary~\ref{cortoProp41},
$\Lambda/ 2 \Lambda \cong (\Z / 2 \Z)^r$ is generated by $n- \sum_{i=1}^{k} s(i)+b -1$ elements.
Therefore, $r(T) \leq n- \sum_{i=1}^{k} s(i)+b -1$.
\end{proof}

\begin{theorem}\label{boundonr}
Let $T$ be a minimal, very small $(\mathcal{A}*B)$-tree with special points.
Then $r_{\Q}(T) \leq 3n+2k-3-3 \sum_{i=1}^{k} s(i)$. Equality holds only if the action
is simplicial and the only elliptic elements are the $A_i$'s.
\end{theorem}

\begin{proof}
By Corollary~\ref{CorIIIthree} and Corollary~\ref{Cor43}, if $T$ is geometric,
then $r(T) \leq 3n+2k-3-3 \sum_{i=1}^{k} s(i)$. If $T$ is not geometric, there is a
sequence of geometric very small trees $T_{\mathcal{K}_m}$ strongly converging
to $T$. By Proposition~\ref{Prop42}, $r_{\Q}(T) < 3n+2k-3-3 \sum_{i=1}^{k} s(i)$.
We conclude the proof of the theorem proving that if the action is geometric,
but there are other elliptic elements beside the $A_i$'s, then
$\Lambda/ 2 \Lambda \cong (\Z / 2 \Z)^r$ with $r < 3n+2k-3-3 \sum_{i=1}^{k} s(i)$.
There are three different cases.
\begin{enumerate}
\item If the action is simplicial, then it is obtained from a graph of groups $\Gamma$.
Consider the natural (topological) epimorphism $\rho: \mathcal{A}*B \rightarrow \pi_1(\Gamma)$.
Since there are other elliptic elements beside the $A_i$'s, some vertex group is nontrivial
and $\rho$ is not injective. Because the free groups are Hopfian,
$\mbox{rk}(\Gamma) < n - \sum_{i=1}^{k} s(i)$.
On the other hand, since there are not inversions, every vertex of $\Gamma$
is the projection of a branch point of $T$. By Corollary~\ref{CorIIIthree},
$\Gamma$ has at most $2n + 2k - 2 - 2 \sum_{i=1}^{k} s(i)$ vertices. Therefore,
$\Gamma$ has strictly less than $3n+2k-3-3 \sum_{i=1}^{k} s(i)$ edges. Because $\Lambda$
is generated by lengths of edges, $r< 3n+2k-3-3 \sum_{i=1}^{k} s(i)$.

\item Suppose that every $(\mathcal{A}*B)$-orbit is dense in $T$. In the
first case, we had $r< 3n+2k-3-3 \sum_{i=1}^{k} s(i)$ because $L/2L$ had $2$-rank
$<n - \sum_{i=1}^{k} s(i)$. In this case, we prove that $\Lambda / L$
has $2$-rank $< 2n+2k-3 -2 \sum_{i=1}^{k} s(i)$,
so that $\Lambda/ 2 \Lambda$ has $2$-rank $< 3n+2k-3-3 \sum_{i=1}^{k} s(i)$.
Since $T$ is geometric, $T=T_{\mathcal{K}}$, and we can suppose that every terminal vertex of $K$
is a branch point in $T$. By Corollary~\ref{cortoProp41}, if the
number of orbits of branch points is $< 2n + 2k - 2 - 2 \sum_{i=1}^{k} s(i)$,
then we are done. Otherwise, let
$$
p_1, \ldots, p_{2n + 2k - 2 - 2 \sum_{i=1}^{k} s(i)}
$$
be representatives of these orbits belonging to $K$. Each $p_j$ has index $1$.
By Proposition~\ref{PropI8andCor}, every edge stabilizer is trivial.
Therefore, the generators $\{ \varphi_i, \varphi^j \}$ are \emph{independent}
in the sense of \cite{GLP}: a reduced word cannot be equal to the identity on
a non-degenerate subinterval of $K$. Denoting by $|\cdot|$ the arclength in $K$,
we have
$$
|K| = \sum_{i=1}^{n - \sum s(j)} |P_i| + \sum_{j=1}^{k} |Q_j|
$$
(see \cite{GaLe} and \cite{GLP}). This gives an equality between numbers of the
form $d(q,q')$, where $qq'$ is an edge of $K$ or $P_i$ or $Q_j$. Since every vertex of
$K$, and so of $P_i$ and $Q_j$, is a branch point of $T$, this is an equation in $\Lambda/L$.
Now, because we have

$$
\left.
  \begin{array}{ccl}
    d(q,q') & = & d(gp_j, hp_m) \\
     & = & d(gp_j, gp_1) + d(gp_1, gp_m) + d(gp_m, hp_m) \\
     & = & d(p_j, p_1) + d(p_1, p_m) + l(g^{-1} h) \\
     & = & d(p_1, p_j) + d(p_1, p_m) \mbox{ mod } L, \\
  \end{array}
\right.
$$

we can replace $d(q,q')$ by $d(p_1, p_j) + d(p_1, p_m)$. Hence, we have a relation
between the elements of $\{ d(p_1, p_j) \, | \, j=2, \ldots, 2n + 2k - 2 - 2 \sum_{i=1}^{k} s(i) \}$,
whose coefficients are integers mod $2$. We need to show that this relation is
not trivial. The coefficient of $d(p_1, p_j)$ in the expansion of $|K|$
has the same parity as $\sum v_K(x)$ taken over vertices of $K$ belonging
to the orbit of $p_j$, where $v_K(x)$ is the valence of $x$ in $K$.
The coefficient of $d(p_1, p_j)$ in the expansion of $|P_i|$
(respectively $|Q_j|$) has the same parity as $\sum v_{P_i}(x)$ (respectively $\sum v_{Q_j} (x)$)
taken over vertices of $P_i$ (respectively $Q_j$) belonging to the orbit of $p_j$, where $v_{P_i}(x)$
is the valence of $x$ in $P_i$ (and $v_{Q_j}(x)$ is the valence of $x$ in $Q_j$).
Since every $p_j$ has index $1$, as in the proof of Theorem~\ref{thmIII2},
we have the nontrivial relation between the generators of $\Lambda/L$:

$$
\sum_{j=2}^{2n + 2k - 2 - 2 \sum_{i=1}^{k} s(i)} d(p_1,p_j) = 0 \mbox{ mod } L.
$$

\item Assume that the action is not simplicial. Therefore, $T$ may be obtained as a
graph of transitive action (see \cite{Le}). In particular, there exists a subtree
$T_v$ in $T$ such that

\begin{itemize}
\item $T_v$ is closed and not equal to a point;

\item there is $\delta >0$ such that, for $g \in \mathcal{A}*B$,
$g \in \mbox{Stab}(T_v)$ or the distance between $T_v$ and $gT_v$
is greater than $\delta$;

\item $\mbox{Stab}(T_v)$ acts on $T_v$ with dense orbits.
\end{itemize}

Let $T'$ be the $(\mathcal{A}*B)$-tree obtained by collapsing each $g T_v$ to
a point. The natural action of $\mathcal{A}*B$ on $T'$ is very small.
By Theorem~\ref{thmIII2} applied to both $T$ and $T'$, we have
$\mbox{rk} (\mbox{Stab}(T_v)) = m < \infty$ and if $i(T_v)$ is the
index of $T_v$ with respect to a $\mbox{Stab}(T_v)$-tree,
$$
0 \leq i(T) - i(T') = i(T_v) - \Big( 2m+2k-2 -2 \sum_{i=1}^{k} s(i) \Big) \leq 0,
$$
where the first inequality is coming from the fact that $T$ is geometric.
Hence, $i(T_v) = 2m+2k-2 -2 \sum_{i=1}^{k} s(i)$ and the action
of $\mbox{Stab}(T_v)$ on $T_v$ is geometric. If there are less than
$2m+2k-2 -2 \sum_{i=1}^{k} s(i)$ distinct $\mbox{Stab}(T_v)$-orbits
of branch points in $T_v$, then there are less than
$2n+2k-2 -2 \sum_{i=1}^{k} s(i)$ distinct
$(\mathcal{A}*B)$-orbits in $T$, and we are done. Otherwise,
as in the previous case, we have a nontrivial relation in $\Lambda(T_v)/L(T_v)$,
and hence in $\Lambda(T)/L(T)$.
\end{enumerate}
\end{proof}

Recall that our goal is to prove the following theorem.
\begin{Theorem}
The dimension of $\partial \mathrm{CV}_n(\mathcal{A})$ is equal to
$\mathrm{dim} (\mathrm{CV}_n(\mathcal{A}))-1$.
\end{Theorem}
\begin{proof}
We know that $\partial \mathrm{CV}_n(\mathcal{A})$ is the set of projective classes
of lengths functions of very small actions of $F_n$ such that the $A_i$'s are elliptic,
but those are not the only elliptic elements. By Proposition V.1. \cite{GaLe}, if $G$
is a finitely generated group, the space of all projectivized length functions with
$\Q \mbox{-rank} \leq N$ has (topological) dimension $\leq N-1$. By Theorem~\ref{boundonr},
the $\Q$-rank in the case $G= \mathrm{Out}(F_n;\mathcal{A})$ is
$\leq \mbox{dim}(\mathrm{CV}_n(\mathcal{A}))$. Therefore, $\mbox{dim} (\partial \mathrm{CV}_n(\mathcal{A}))
\leq 3n + 2k - 5 - 3 \sum_{i=1}^{k} s(i) = \mathrm{dim} (\mathrm{CV}_n(\mathcal{A}))-1$.
Because it is very easy to find a simplex in $\partial \mathrm{CV}_n(\mathcal{A})$ of dimension
$3n + 2k - 5 - 3 \sum_{i=1}^{k} s(i)$, we conclude that $\mbox{dim} (\partial \mathrm{CV}_n(\mathcal{A}))
= \mathrm{dim} (\mathrm{CV}_n(\mathcal{A}))-1$.
\end{proof}

\begin{example}
In Example~\ref{expoly}, we described $\mathrm{CV}_{2}(A)$, where $F_2 = <a, \, b>$,
$A=<a>$. Obviously, $\mbox{dim} (\partial \mathrm{CV}_2(A)) =0
= \mathrm{dim} \mathrm{CV}_2(A)-1$. Note that in this case the boundary is not connected.
\end{example}
In conclusion, we computed the dimension of $\partial \mathrm{CV}_n(\mathcal{A})$
using the $\Q$-rank of a tree with special points.

\end{document}